\documentclass[a4paper,11pt]{article}
\usepackage{amsmath}
\usepackage{amsthm}
\usepackage[all]{xy}
\usepackage[english]{babel}
\usepackage[utf8]{inputenc}
\usepackage{graphicx}
\usepackage{xypic,color, xcolor}
\usepackage{amsfonts, amssymb}

\usepackage{utfsym}
\usepackage{tikz-cd}
\newtheorem{thm}{Theorem}
\newtheorem{defn}[thm]{\textbf{Definition}}
\newtheorem{example}{Example}[section]

\newtheorem{cor}[thm]{Corollary}

\newtheorem{prop}[thm]{Proposition}
\theoremstyle{definition}

\theoremstyle{definition}

\theoremstyle{remark}
\newtheorem{rem}[thm]{Remark}

\renewcommand{\k}{\Bbbk}

\newcommand{\de}{\Delta}

\newcommand{\ot}{\otimes}

\newcommand{\rt}{\rightarrow}

\usepackage{xypic,color}
\usepackage{amsfonts, amssymb}%
\usepackage{multicol}
\usepackage{amscd, amsthm, mathrsfs}
\usepackage{todonotes}

\usepackage{rotating}
\usetikzlibrary{calc}

\tikzset{curve/.style={settings={#1},to path={(\tikztostart)
			.. controls ($(\tikztostart)!\pv{pos}!(\tikztotarget)!\pv{height}!270:(\tikztotarget)$)
			and ($(\tikztostart)!1-\pv{pos}!(\tikztotarget)!\pv{height}!270:(\tikztotarget)$)
			.. (\tikztotarget)\tikztonodes}},
	settings/.code={\tikzset{quiver/.cd,#1}
		\def\pv##1{\pgfkeysvalueof{/tikz/quiver/##1}}},
	quiver/.cd,pos/.initial=0.35,height/.initial=0}

\begin{document}

\title{Bounds on Frobenius dimension}
\author{D. Artenstein, J. Cóppola, J. Finot, A. Gonz\'alez and G. Mata}
\date{\today}
\maketitle


%
%
%
%
%



\begin{abstract}

In this article, we prove two types of results about the Frobenius dimension of associative algebras. First, we refine the known upper bound for the Frobenius dimension of a finite dimensional algebra  in terms of its dimension as a vector space, and show that the only algebras reaching this bound are radical square zero algebras associated with single-vertex quivers. Second, we compute Frobenius dimension for low-dimensional algebras explicitly, and for truncated path algebras in terms of paths with no detours in their underlying quivers.

2010 Mathematics Subject Classification 16W99.
\end{abstract}

\section{Introduction}

The concept of a nearly Frobenius algebra originates from a topological motivation: the fact that the homology of the free loop space ($H_*(LM)$) of a manifold possesses a coproduct structure similar to that of a Frobenius algebra, yet lacks a defined counit (\cite{CG}). This result has shifted interest toward an algebraic analysis of these structures, viewed as a natural generalization of classical Frobenius algebras,  studied in \cite{GLSU},  \cite{AGL15}, \cite{AGM20} and \cite{AGM22} among others.
 
Central to this theory is the Frobenius space ($\mathscr{E}_A$), which encompasses all possible coproducts that endow an algebra A with a nearly Frobenius structure (\cite{AGL15}). The dimension of this space, called the Frobenius dimension ($\operatorname{Frobdim} (A)$), serves as an algebraic invariant. While it is well-established that for Frobenius algebras this dimension coincides with the dimension of the algebra itself, its behavior in the general case remains a fertile and complex field of study.

This article is structured around two central pillars that significantly advance the understanding of this invariant. First, we address the global limits of the Frobenius dimension. Existing literature established a general upper bound of $(dim_{\k}A)^2$, a limit known to be loose since equality is never achieved for algebras of dimension greater than one.  In this work, we refine this bound by proving that for any finite-dimensional algebra of dimension $n\geq 3$, the Frobenius dimension is bounded by $(n-1)^2.$ 
\begin{thm}\label{1}(Theorem \ref{teo7})
	If $A$ is a $\k$-algebra of finite dimension $n\geq 3$, then $$\operatorname{Frobdim}(A)\leq (n-1)^2.$$
\end{thm}
Furthermore, we provide a complete characterization of the unique family of algebras that reach this maximum value: the radical square zero algebras associated with single-vertex quivers.
\begin{thm}(Proposition \ref{prop radical^2}, Theorem \ref{frobdim maxima})
	If $A$ is a $\k$-algebra of finite dimension $n\geq 3$, the following statements are equivalent
	
	\begin{enumerate}
		\item $\operatorname{Frobdim}(A)=(n-1)^2$. 
		\item $A$ admits a basis, as $\k$-vector space, $\mathcal{B}=\{1,v_2,\dots,v_n\}$: $v_iv_j=0$ for all $i,j=2,\dots,n$.
		\item  $A$ is a radical square  zero algebra with quiver 
		$$Q:
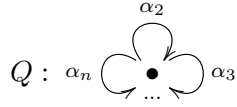
\xymatrix{
			\underset{\dots}{\bullet}  \ar@(ul,dl)_{\alpha_n}\ar@(ur,dr)[]^{\alpha_3}\ar@(ul, ur)[]^{\alpha_2}
		}$$
		
	\end{enumerate}
	\end{thm}
	
Second, this paper conducts an exhaustive investigation into truncated path algebras. Moving beyond previous descriptions that were computationally difficult to implement, we provide explicit formulas for calculating both nearly Frobenius coproducts and the exact Frobenius dimension in these settings. The strength of this analysis lies in the introduction of two technical concepts: paths with ``no detours'' and the ``chop'' operation. These tools not only reduce the underlying technical complexity but also allow for proofs that are elegant and straightforward to write. 

\begin{thm}(Theorem \ref{frobdim truncada LLm})
	Let $\displaystyle{A=\k Q/R^m}$ be a truncated path $\Bbbk$-algebra,  and $Q$ a connected quiver. 
	Then, the general expression of the nearly Frobenius coproduct is:
	$$\Delta(1)=\sum_{i=-1}^{m-2}\quad\sum_{\alpha_1\cdots\alpha_{m+i}\in \mathcal{ND}_{m+i}}\lambda_{\alpha_1\cdots\alpha_{m+i}}w_{\alpha_1\cdots\alpha_{m+i}}$$
	where $\lambda_{\alpha_1\cdots\alpha_{m+i}}\in\k$,  and $w_{\alpha_1\cdots\alpha_{m+i}}$ is the chop of  $\alpha_1\cdots\alpha_{m+i}$, for $i=-1,\dots, m-2$.\\
	In particular,
	\begin{equation}\label{e2}
		\operatorname{Frobdim}_\k(A)= \sum_{i=-1}^{m-2}|\mathcal{ND}_{m+i}|=\sum_{i=-1}^{m-3}|\mathcal{ND}_{m+i}|+|Q_{2m-2}|	
	\end{equation}
\end{thm}

Consequently, this work offers both a panoramic view of the theoretical limits of the invariant and a precise toolkit for its calculation in specific families of algebras.\\

The manuscript is organized as follows: Section 2 establishes the theoretical foundation by  fundamental concepts of path algebras and nearly Frobenius algebras. It also introduces an intermediate degree—defined by the number of paths of a specific length starting or ending at a vertex—and explores its relationship to the in-degree and out-degree of quivers. In Section 3, we present a comprehensive classification of low-dimensional nearly Frobenius algebras (specifically for dimensions 2 and 3), determining their explicit coproducts and Frobenius dimensions for both commutative and non-commutative cases.
Section 4 is dedicated to the core theoretical contribution: the refinement of the general upper bound for the Frobenius dimension. We prove that for any algebra of dimension $n\geq 3$, the dimension is bounded by $(n-1)^2$, and we characterize the specific radical square zero algebras that reach this limit. 
This analysis is extended in Section 5, where we specialize these bounds for bound quiver algebras and connected quivers with multiple vertices, providing even sharper constraints based on the combinatorial structure of the quiver. Finally, Section 6 focuses on truncated path algebras, where we introduce the concepts of paths with ``no detours'' and the ``chop'' operation. These tools allow us to derive explicit formulas for nearly Frobenius coproducts and determine the exact Frobenius dimension in these complex settings.

\bigbreak


\section{Preliminaries}
This section establishes the necessary theoretical framework by reviewing the fundamental properties of path algebras and nearly Frobenius structures.
We also introduce an intermediate degree for vertices based on paths of a fixed length, providing a tool to link the combinatorial properties of quivers with the algebraic behavior of their associated algebras.

\subsection{Path algebras}


\begin{defn}
	A {\bf quiver $Q = \bigl(Q_0,Q_1, s, t\bigr)$} is a quadruple consisting of:
	\begin{itemize}
	\item two sets $Q_0$ (whose elements are called vertices)  and $Q_1$ (whose elements are called arrows), 
	\item two maps $s,t:Q_1\to Q_0$ which associate to each arrow $\alpha\in Q_1$ its {\bf source} $s(\alpha)$ and its {\bf target} $t(\alpha)$. 
	\end{itemize}

We denote an arrow of source $a=s(\alpha)$ and target $b=t(\alpha)$ by $\alpha: a \to b$, and a quiver $Q = \bigl(Q_0,Q_1, s, t\bigr)$ simply by $Q$.

\end{defn}

\begin{defn}
	Let $Q = \bigl(Q_0,Q_1, s, t\bigr)$ be a quiver and $a,b\in Q_0$. A {\bf path of length} $l\geq 1$ with source $a$ and target $b$ is a sequence of arrows $\alpha_1\alpha_2\cdots \alpha_l$ such that 
	$s(\alpha_1)=a$, $t(\alpha_k)=s(\alpha_{k+1})$ for each $1 \leq k < l$  and $t(\alpha_l)=b$.

\end{defn}

We denote by $Q_k$ the set of paths of lenght $k \in \mathbb{N}$.
We denote by $C^k$ the cyclic quiver of $k$ vertices. 

\begin{defn}
	Let $Q$  be a quiver. The {\bf path algebra}  $\Bbbk Q$ of $Q$ is the $\k$-algebra whose underlying $\k$-vector space has as its basis the set of all paths of length $l\geq 0$ in $Q$. \\

The product of two paths $\alpha_1\dots\alpha_l$  and $\beta_1\dots\beta_k$
is equal to zero if $t(\alpha_l)\neq s(\beta_1)$ and is equal to
the composed path $\alpha_1\dots\alpha_l\beta_1\dots\beta_k$ if $t(\alpha_l)=s(\beta_1)$. The product of basis elements is, then, extended to arbitrary elements of $\Bbbk Q$ by distributivity.
\end{defn}

\begin{defn}
	Let $Q$ be a quiver and $\Bbbk Q$ the associated path algebra. The {\bf arrow ideal} of $\Bbbk Q$, denoted by $R_Q$, is the two-sided ideal in $\Bbbk Q$ generated by all arrows.
\end{defn}

For any $m \geq 1$, we have that $R_Q^m$ is the two-sided ideal generated by all paths of length $m$ and we have the following chain of ideals: $$R_Q\supseteq R_Q^2 \supseteq R^3_Q \supseteq R_Q^4 \supseteq \dots$$

\begin{defn}
	Let $Q$ be a finite quiver and $R_Q$. A two-sided ideal $I$ of $\k Q$ is said to be {\bf admissible} if there exists $m\geq 2$ such that
	$$R^m_Q\subseteq I\subseteq R^2_Q.$$
	If $I$ is an admissible ideal of $\k Q$, the pair $(Q,I)$ is said to be a {\bf bound quiver.} The quotient algebra $\k Q/I$ is said to be the algebra of the bound quiver $(Q,I)$ or, simply, a {\bf bound quiver algebra}. 
\end{defn}

\begin{defn}
	An algebra $\Bbbk Q/I$ is a {\bf connected algebra} if  $Q$ is a connected quiver, where $Q$ is said to be connected if the underlying
	graph of $Q$ is a connected graph.
\end{defn}

\begin{defn}
	Let $Q$ be a quiver. A {\bf relation} in $Q$ with coefficients in $\k$ is a $\k$-linear combination of paths of length at least two having the same source and target. Thus, a relation $\rho$ is an element of $\k Q$ such that 
	$$\rho=\sum_{i=1}^h\lambda_iw_i,$$
	where $\lambda_i\in \k$ are not all zero and each $w_i$ is a path in $Q$ of length at least $2$ such that $s(w_i) = s(w_j)$ and $t(w_i)=t(w_j)$ for all $i \neq j$.

If $h=1$, the preceding relation is called a {\bf monomial relation}. The algebra $\Bbbk Q/I$ is said to be {\bf monomial} if the ideal $I$ can be generated by monomial relations.
\end{defn}

\begin{defn}
	A {\bf truncated path algebra} is a monomial algebra of the form $\k Q/I$ where $I\subset \k Q$ is the ideal generated by all paths of a fixed length $m$. For $m = 2$ we say that $\k Q/I$ is a radical square zero algebra.
\end{defn}

\begin{defn}\label{def3}
	Let $Q$ be a quiver and $v \in Q_0$ a vertex. The {\bf out-degree} of $v$, denoted by $gr^+(v)$, is the number of arrows leaving $v$ and the {\bf in-degree} of $v$, denoted by $gr^-(v)$, is the number of arrows entering the vertex $v$. 
\end{defn}

\begin{defn}
	A quiver $Q$ is said to be {\bf regular} if $$gr^-(v) = gr^+(v) = gr^+(w) = gr^-(w)$$ for all $v,w \in Q_0$.  In this case we denote by $gr(Q) = gr^+(v)$ for any $v \in Q_0$
\end{defn}

\begin{defn}
	Given an algebra $\Bbbk Q/I$, we define $g\mathcal{P}^{+}(v)$ as the dimension of the subspace of $\Bbbk Q/I$ generated by the paths starting in $v$ and $g\mathcal{P}^{-}(v)$ as the dimension of the subspace of $\Bbbk Q/I$ generated by the paths ending in $v$.
\end{defn}

\begin{rem}
%

If $A$ is a truncated path algebra where $I = R^k$ and $Q$ is regular, then 

$$ g\mathcal{P}^{+}(v) = \sum_{l=1}^{k-1} gr^+(v)^l = g\mathcal{P}^{-}(v). $$

\end{rem}


\subsection{Nearly Frobenius algebras}

In \cite{A96} Abrams proves a theorem that gives a characterization of  Frobenius algebras in terms of the existence of a coassociative counital comultiplication $\Delta:A\rightarrow A\otimes A$ which is a map of $A$-bimodules. In his construction, $\Delta(1)$ can be expressed from a basis and the dual one related to the bilinear form of the Frobenius algebras. Let us recall Abrams theorem.

\begin{thm}\label{TM Abrams}
	An algebra $A$ is a Frobenius algebra if and only if it has a coassociative counital comultiplication $\Delta:A\rightarrow A\otimes A$ which is a map of $A$-bimodules.
\end{thm}
We will now recall the definition of nearly Frobenius algebras that is a weakening of Frobenius algebras in the sense of Abrams characterization.

\begin{defn}
	An associative $\k$-algebra $A$ is a {\bf nearly Frobenius algebra} if it admits an homomorphism $\Delta:A\rt A\ot A$ of $A$-bimodules.

	Let $(A,\Delta_A)$ and $(B,\Delta_B)$ be two nearly Frobenius algebras. An homomorphism $f:A\rightarrow B$ is a {\bf nearly Frobenius homomorphism} if it is a morphism of $\k$-algebras and the following diagram commutes.
	$$\xymatrix{A\ar[r]^{f}\ar[d]_{\Delta_A}& B\ar[d]^{\Delta_B}\\
	A\otimes_R A\ar[r]_{f\otimes f}&B\otimes_R B	}.$$

\end{defn}

\begin{defn}
	 The {\bf Frobenius space} associated to an algebra $A$ is the vector space of all the possible
	coproducts $\Delta$ that make it into a nearly Frobenius algebra $\mathscr{E}_A$. Its dimension over $\k$
	is called  \emph{Frobenius dimension} of $A$, that is,
	$$\operatorname{Frobdim}A=\operatorname{dim}_\k\mathscr{E}_A$$
\end{defn}

In Corollary 10 of \cite{{AGM22}} the particular case where $A$ is a Frobenius algebra is considered.\\
It is proven that if $A$ is a Frobenius algebra, then $$\operatorname{dim}_\k A=\operatorname{Frobdim}A.$$
As a straightforward consequence, the computation of the $\operatorname{Frobdim}A$ can be a useful tool to conclude that an algebra is not Frobenius.

\section{Classification of low-dimensional nearly Frobenius algebras }
In this section, we provide a systematic classification of nearly Frobenius structures in dimensions 2 and 3. By determining their explicit coproducts and Frobenius dimensions, we illustrate the variety of behaviors this invariant can exhibit in both commutative and non-commutative settings, providing a catalog of examples that motivate the general bounds discussed later.
\subsection{Dimension two:} In this case $A$ is a commutative algebra and we can choose $\mathscr{B}=\{1,v\}$ basis of $A$.
Let $\Delta$ be a possible nearly Frobenius coproduct $$\Delta(1)=x_1 1\otimes 1+x_2 1\otimes v+x_3 v\otimes 1 +x_4 v\otimes v.$$
We write $v^2=a1+bv$, for $a,b\in\k$. Using that $\Delta$ is a bimodule morphism and applying it in $v$ we obtain the following equations
	$$\begin{array}{rcl}
		ax_3-ax_2&=&0\\
		x_1+bx_2-ax_4&=&0\\
		x_1+bx_3-ax_4&=&0\\
		x_2-x_3&=&0
	\end{array}$$
	Then
	$$\Delta(1)=x_4\bigl(a1\otimes 1+v\otimes v\bigr)+x_3\bigl(1\otimes v+v\otimes 1-b1\otimes 1\bigr).$$
	Therefore $\Delta_1(1)=a1\otimes 1+v\otimes v$ and $\Delta_2(1)=1\otimes v+v\otimes 1-b1\otimes 1$ form a basis of $\mathscr{E}_A$, in particular $\operatorname{Frobdim}_\k(A)=2$.\\
	
	Note that if we take $\Delta(1)=\alpha\bigl(a1\otimes 1+v\otimes v\bigr)$ and define $\varepsilon:A\to\k$ as $\displaystyle{\varepsilon(1)=\frac{1}{a\alpha}}$ and $\varepsilon(v)=0$ then $(A,\varepsilon)$ is a Frobenius algebra for all $a,\alpha\in\k^{\times}$.

\subsection{Dimension three:}
	We will divide this problem into two cases.
	\begin{enumerate}
		\item[(a)] Suppose that there is a basis $\mathcal{B}=\bigl\{1,v,v^2\bigr\}$ of $A$. In particular, this implies that $A$ is a commutative algebra. Let us write $v^3=a1+bv+cv^2$, with $a,b,c\in\k$ and consider  a nearly Frobenius coproduct 
		$$\Delta(1)=x_1 1\otimes 1+x_2 1\otimes v+x_3 v\otimes 1+ x_4 v\otimes v+ x_5 1\otimes v^2 +x_6 v^2\otimes 1$$
		$$x_7 v\otimes v^2 + x_8 v^2\otimes v + x_9 v^2\otimes v^2,$$
		where $x_1,\dots, x_9$ are the variables to be determined.\\
		Since the coproduct $\Delta$ is a bimodule morphism we obtain the following equations that relate its coefficients.
		\begin{multicols}{2}
			\begin{equation}
				ax_5-ax_6=0
			\end{equation}
			\begin{equation}
				x_1+bx_5-ax_8=0
			\end{equation}
			\begin{equation}
				x_1+bx_6-ax_7=0
			\end{equation}
			\begin{equation}
				x_2-x_3-bx_7+bx_8=0
			\end{equation}
			\begin{equation}
				x_2+cx_5-ax_9=0
			\end{equation}
			\begin{equation}
				x_3+cx_6-ax_9=0
			\end{equation}
			\begin{equation}
				x_4-x_5+cx_7-bx_9=0
			\end{equation}
			\begin{equation}
				x_4-x_6+cx_8-bx_9=0
			\end{equation}
			\begin{equation}
				x_7-x_8=0
			\end{equation}
		\end{multicols}
		
		Solving the linear system we obtain that
		
		$$\left\{\begin{array}{rcl}
			x_1&=&-bx_6+ax_8\\
			x_2&=&-cx_6+ax_9\\
			x_3&=&-cx_6+ax_9\\
			x_4&=&x_6-cx_8+bx_9\\
			x_5&=&x_6\\
			x_7&=&x_8
		\end{array}\right.$$
		Then
		$$ \Delta(1)=x_6\bigl(-b 1\otimes 1-a\bigl(1\otimes v+v\otimes 1\bigr)+v\otimes v+1\otimes v^2+v^2\otimes 1\bigr)+x_8\bigl(a1\otimes 1-cv\otimes v+v\otimes v^2+v^2\otimes v\bigr)$$
		$$+x_9\bigl(a\bigl(1\otimes v+v\otimes 1\bigr)+bv\otimes v+v^2\otimes v^2\bigr)$$
		In particular  $\operatorname{Frobdim}_\k(A)=3$.
		
		\item[(b)] Next, suppose that $1$, $u$ and $u^2$ are linearly dependent for any $u\in A$. Let $\bigl\{1,v,w\bigr\}$ be a linear basis of $A$, then, if we write $u=\alpha 1+\beta v+\gamma w$ and $u^2=\alpha'1+\beta'v+\gamma'w$, it must be fulfilled that $\beta\gamma'-\gamma\beta'=0$, for all $\beta,\gamma\in\k$. \\
		On the other hand, if we write $v^2=a'1+av$, $w^2=b'1+bw$, $vw=c'+c_1v+c_2w$ and $wv=d'1+d_1v+d_2w$ with all coefficient in $\k$, using the associativity of the product, we have that:
		$$\left\{\begin{array}{rcl}
			a'&=& c_2(c_2-a)=d_2(d_2-a)\\
			b'&=& c_1(c_1-b)=d_1(d_1-b)\\
			c'&=& -c_1c_2\\
			d' &=& -d_1d_2
		\end{array}\right.$$
		Finally, writing $u^2$ in terms of $\alpha$, $\beta$ and $\gamma$ we have that 
		$$\beta\gamma'-\gamma\beta'=((c_2+d_2)-a)\beta^2\gamma+(b-(c_1+d_1))\beta\gamma^2=0,\;\forall \beta,\alpha\in\k \Rightarrow\left\{\begin{array}{rcl}a&=&c_2+d_2\\b&=&c_1+d_1\end{array}\right.$$
		Then the product is determined by the following relations
		$$\begin{array}{rcl}
			v^2&=&-c_2d_21+(c_2+d_2)v\\
			w^2&=&-c_1d_11+(c_1+d_1)w\\
			vw&=&-c_1c_21+c_1v+c_2w\\
			wv&=&-d_1d_21+d_1v+d_2w
		\end{array}$$
		
		This allows us to establish the equations associated with the nearly Frobenius coproduct
		$$\Delta(1)=x_1 1\otimes 1+x_2 1\otimes v+x_3 1\otimes w+ x_4 v\otimes 1+ x_5 v\otimes v +x_6 v\otimes w$$
		$$x_7 w\otimes 1 + x_8 w\otimes v + x_9 w\otimes w$$
		$\Delta(v)=v\Delta(1)=\Delta(1)v$ implies\\
		$$\begin{array}{l}
			\displaystyle{c_2d_2x_2+d_1d_2x_3-c_2d_2x_4-c_1c_2x_7}=0\\
			\displaystyle{	x_1+(c_2+d_2)x_2+d_1x_3+c_2d_2x_5+c_1c_2x_8}=0\\
			\displaystyle{	d_2x_3+c_2d_2x_6+c_1c_2x_9}=0\\
			\displaystyle{x_1+(c_2+d_2)x_4+c_2d_2x_5+d_1d_2x_6+c_1x_7}=0\\
			\displaystyle{x_2-x_4-d_1x_6+c_1x_8}=0\\
			\displaystyle{x_3+c_2x_6+c_1x_9}=0\\
			\displaystyle{c_2x_7+c_2d_2x_8+d_1d_2x_9}=0\\
			\displaystyle{x_7+d_2x_8+d_1x_9}=0\\
			\displaystyle{(c_2-d_2)x_9}=0\\
		\end{array}$$
		On the other hand, $\Delta(w)=w\Delta(1)=\Delta(1)w$ implies\\
		$$\begin{array}{l}
			\displaystyle{c_1c_2x_2+c_1d_1x_3-d_1d_2x_4-c_1d_1x_7}=0\\
			\displaystyle{x_1+c_2x_2+(c_1+d_1)x_3+d_1d_2x_6+c_1d_1x_9}=0\\
			\displaystyle{c_1x_2+d_1d_2x_5+c_1d_1x_8}=0\\
			\displaystyle{d_1x_4+c_1c_2x_5+c_1d_1x_6}=0\\
			\displaystyle{(c_1-d_1)x_5}=0\\
			\displaystyle{x_4+c_2x_5+c_1x_6}=0\\
			\displaystyle{x_1+d_2x_4+(c_1+d_1)x_7+c_1c_2x_8+c_1d_1x_9}=0\\
			\displaystyle{x_3+d_2x_6-x_7-c_2x_8}=0\\
			\displaystyle{x_2+d_2x_5+d_1x_8}=0
		\end{array}$$
	$\bullet$	In the commutative case, that is $c_1=d_1$ and $c_2=d_2$, the relationship between the coefficients is as follows
		$$\left\{\begin{array}{rcl}
			x_1&=&c_2^2(x_5+x_9)+c_1c_2(x_6+x_8)\\
			x_2&=&-(c_2x_5+c_1x_8)\\
			x_3&=&-(c_2x_6+c_1x_9)\\
			x_4&=&-(c_2x_5+c_1x_6)\\
			x_7&=&-(c_2x_8+c_1x_9)
		\end{array}\right.$$
		Then
		$$\Delta(1)=x_5 \bigl(c_2^21\otimes 1-c_21\otimes v-c_2v\otimes 1+v\otimes v\bigr)+x_6\bigl(c_1c_21\otimes 1+v\otimes w-c_21\otimes w-c_1v\otimes 1\bigr)$$$$+x_8\bigl(c_1c_21\otimes 1-c_11\otimes v+w\otimes v-c_2w\otimes 1\bigr)+ x_9\bigl(c_2^21\otimes 1-c_11\otimes w-c_1w\otimes 1+w\otimes w\bigr).$$
		In particular  $\operatorname{Frobdim}_\k(A)=4$.\\
		
		Note that this family of algebras admits a basis of the form $\{1, v', w'\}$ where $(v')^2=(w')^2 = v'w'= 0$. It suffices to make the change of variables 
		$$\left\{\begin{array}{c}
			v' = v - c_21\\
			w' = w - c_11,
		\end{array}\right.$$ and this implies that $$v'w' = (v-c_21)(w-c_11)=vw-c_1v-c_2w+c_1c_21=0.$$   This is in accordance with the classification described in \cite{KST21}.\\
		
		
	$\bullet$	In the noncommutative case, that is $c_1\neq d_1$ (and) or $c_2\neq d_2$, the relationships between the coefficients are as follows
		$$\left\{\begin{array}{rcl}
			x_1&=&(c_1c_2-d_1d_2)x_6\\
			x_2&=&-d_1x_8\\
			x_3&=&-c_2x_6\\
			x_4&=&-c_1x_6\\
			x_5&=&0\\
			x_6&=&-x_8\\
			x_7&=&-d_2x_8\\
			x_9&=&0
		\end{array}\right.$$
		Then
		$$\Delta(1)=x_6\bigl[(c_1c_2-d_1d_2)1\otimes 1+d_11\otimes v-c_21\otimes w-c_1v\otimes 1+v\otimes w+d_2w\otimes 1-w\otimes v\bigr]$$
		In particular  $\operatorname{Frobdim}_\k(A)=1$.\\
		
		As in the previous case we can make a change of variables such that in the new basis $\{1,u, u'\}$ the product of $A$ is $u^2=1$, $(u')^2=0$, $uu'=u'$ and $u'u=-u'$.\\
		If $c_1=d_1$ and $c_2\neq d_2$ we define 
		$$\left\{\begin{array}{lcl}
			v'=v-\frac{c_2+d_2}{2}1&\Rightarrow &(v')^2=\left(\frac{c_2-d_2}{2}\right)^2\Rightarrow u=\frac{2v'}{c_2-d_2}\Rightarrow u^2=1\\
			u'=w-c_11&\Rightarrow& (u')^2=0
		\end{array}\right.$$
		With this change of variables, we can see that  $uu'= u'$ and $u'u=-u'$.	
		
		If $c_1\neq d_1$ and $c_2= d_2$ we define 
		$$\left\{\begin{array}{lcl}
			u'=v-c_21&\Rightarrow& (u')^2=0\\
			w'=w-\frac{c_1+d_1}{2}1&\Rightarrow &(w')^2=\left(\frac{c_1-d_1}{2}\right)^2\Rightarrow u=\frac{2w'}{c_1-d_1}\Rightarrow u^2=1
		\end{array}\right.$$
		
		As before, with this change of variable, we can see that  $uu'= u'$ and $u'u=-u'$.	
		
		Finally, if $c_1\neq d_1$ and $c_2\neq d_2$ we define 
		$$\left\{\begin{array}{lcl}
			v'=v-\frac{c_2+d_2}{2}1&\Rightarrow &(v')^2=\left(\frac{c_2-d_2}{2}\right)^2\Rightarrow u=\frac{2v'}{c_2-d_2}\Rightarrow u^2=1\\
			w'=w-\frac{c_1+d_1}{2}1&\Rightarrow &(w')^2=\left(\frac{c_1-d_1}{2}\right)^2\Rightarrow w''=\frac{2w'}{c_1-d_1}\Rightarrow (w'')^2=1
		\end{array}\right.$$
		The last change of variable is $u'=u+w''\Rightarrow (u')^2=0$ and $uu'=u'$.
	\end{enumerate}
\begin{sidewaystable}[p]\label{table 1}
\begin{center}
	\begin{tabular}{| c | c | c |c|}
	\hline
	&&&\\
	$A$ & A basis of $\mathscr{E}_A$ & $\operatorname{Frobdim}(A)$& Frobenius algebra\\ 
	\hline
	&&&\\
	$\operatorname{dim}_\k(A)=2$ & $\{\Delta_1, \Delta_2\}\xrightarrow{b}\mathscr{E}_A$ & &Yes,  \\
	$\{1,v\}\xrightarrow{b}A:$ & $\de_1(1)=a1\ot 1+v\ot v$ & 2& If $a\neq 0$:  $\varepsilon(v)=0$, $\varepsilon(1)=a^{-1}\Rightarrow (A,\Delta_1,\varepsilon)$ FA\\
	$v^2=a1+bv$ & $\de_2(1)=1\ot v+v\ot 1-b1\ot 1$ &&If $a= 0$:  $\varepsilon(1)=0$, $\varepsilon(v)=1 \Rightarrow (A,\Delta_2,\varepsilon)$ FA\\ 
	\hline
	&&&\\
	If $\operatorname{dim}_\k(A)=3$ & $\{\Delta_1, \Delta_2, \de_3\}\xrightarrow{b}\mathscr{E}_A$ && Yes\\
	and $\bigl\{1,v, v^2\bigr\}\xrightarrow{b}A:$ & $\de_1(1)=1\ot v^2+v^2\ot 1+v\ot v-a(v\ot 1+1\ot v)$ & 3&  $\varepsilon(1)=0$, $\varepsilon(v)=0$ and $\varepsilon(v^2)=1$ \\
	$v^3=a1+bv+cv^2$ & $-b1\ot 1$ &&\\
	&$\de_2(1)=v^2\ot v+v\ot v^2-a v\ot v +c1\ot 1$&&$(A,\Delta_1, \varepsilon$) Frobenius algebra\\
	&$\de_3(1)=v^2\ot v^2+bv\ot v+c(v\ot 1+ 1\ot v)$&&\\
	 \hline
	 &&&\\
	 If $\operatorname{dim}_\k(A)=3$ & $\{\Delta_1, \Delta_2, \de_3, \de_4\}\xrightarrow{b}\mathscr{E}_A$ & &No\\
	  $\bigl\{1,u, u^2\bigr\}$ is LD $\forall u\in A$&  & &\\
	 $\{1,v,w\}\xrightarrow{b}A$ & $\de_1(1)=c_2^21\otimes 1-c_21\otimes v-c_2v\otimes 1+v\otimes v$ &4&\\
	 commutative case& $\de_2(1)=c_1c_21\otimes 1+v\otimes w-c_21\otimes w-c_1v\otimes 1$ &&\\
	 $v^2=-c_2^2+2c_2v$ & $\de_3(1)=c_1c_21\otimes 1-c_11\otimes v+w\otimes v-c_2w\otimes 1$ &&$\operatorname{dim}_\k(A)=3<4=\operatorname{Frobdim}(A)$\\
	 $w^2=-c_1^2+2c_1w$& $\de_4(1)=c_2^21\otimes 1-c_11\otimes w-c_1w\otimes 1+w\otimes w$&&\\
	 $vw=-c_1c_21+c_1v+c_2w$&&&\\
	 \hline
	 &&&\\
	 If $\operatorname{dim}_\k(A)=3$ & $\{\Delta_1\}\xrightarrow{b}\mathscr{E}_A$ & &\\
	 $\bigl\{1,u, u^2\bigr\}$ is LD $\forall u\in A$&  & 1& No\\
	 $\{1,v,w\}\xrightarrow{b}A$ & $\de_1(1)=(c_1c_2-d_1d_2)1\otimes 1+d_11\otimes v-c_21\otimes w$ &&\\
	 non commutative case& $-c_1v\otimes 1+v\otimes w+d_2w\otimes 1-w\otimes v$ &&\\
	 $v^2=-c_2d_21+(c_2+d_2)v$ &  &&$\operatorname{Frobdim}(A)=1<3=\operatorname{dim}_\k(A)$\\
	 $w^2=-c_1d_11+(c_1+d_1)w$& &&\\
	 $vw=-c_1c_21+c_1v+c_2w$&&&\\
	 $wv=-d_1d_21+d_1v+d_2w$&&&\\
	 \hline
\end{tabular}
\caption{Classification in low dimension }
\label{tab:clasificación}
\end{center}
\end{sidewaystable}
\newpage
\section{Bounds on Frobenius dimension}

This section is devoted to the core theoretical objective of this work: refining the global bound for the Frobenius dimension. We demonstrate that for any algebra of dimension $n\geq 3$, the Frobenius dimension is strictly bounded by $(n-1)^2$, and we characterize the specific families of algebras that attain this maximum value.

\begin{prop} \label{uv=0}
	Let $A$ be a $\k$-algebra with unit $1$ and $\mathscr{B}=\bigl\{1, v_2, \dots, v_n\bigr\}$ basis of $A$ as a $\k$-vector space such that $v_iv_j=0$ for all $i,j=2, \dots, n$, $n\geq 3$. Then $$\operatorname{Frobdim}_\k(A)=(n-1)^2.$$
	\end{prop}
	\begin{proof}
Let	$\bigl\{1\otimes 1, 1\otimes v_i, v_i\otimes 1, v_i\otimes v_j: i,j=1,\dots, n\bigr\}$ be a basis of $A\otimes A$.	 Then
$$\Delta(1)=\alpha 1\otimes 1+\sum_{i=2}^n\alpha_i1\otimes v_i+\sum_{j=2}^n\beta_jv_j\otimes 1+\sum_{i,j=2}^n\alpha_{ij}v_i\otimes v_j$$
If $\Delta$ defines a nearly Frobenius coproduct then the following relation must be satisfied:
$$v_k\Delta(1)=\Delta(1)v_k$$ 
$$v_k\Delta(1)=\alpha v_k\otimes 1+\sum_{i=2}^n\alpha_iv_k\otimes v_i+\sum_{j=2}^n\beta_jv_kv_j\otimes 1+\sum_{i,j=2}^n\alpha_{ij}v_kv_i\otimes v_j=\alpha v_k\otimes 1+\sum_{i=2}^n\alpha_iv_k\otimes v_i$$
$$\Delta(1)v_k=\alpha 1\otimes v_k+\sum_{i=2}^n\alpha_i1\otimes v_iv_k+\sum_{j=2}^n\beta_jv_j\otimes v_k+\sum_{i,j=2}^n\alpha_{ij}v_i\otimes v_jv_k=\alpha 1\otimes v_k+\sum_{j=2}^n\beta_jv_j\otimes v_k$$
Then $\alpha=0$, $\alpha_i=0$ for all $i\neq k$, $\beta_j=0$ for all $j\neq k$ and $\alpha_k=\beta_k$.\\
Therefore
$$\Delta(1)=\alpha_k\bigl(1\otimes v_k+v_k\otimes 1\bigr)+\sum_{i,j=2}^n\alpha_{ij}v_i\otimes v_j$$
Since $n\geq 3$ there exists $v_h\neq v_k$ then $\alpha_k=0$. Finally
$$\Delta(1)=\sum_{i,j=2}^n\alpha_{ij}v_i\otimes v_j \Rightarrow \operatorname{Frobdim}_\k(A)=(n-1)^2.$$
\end{proof}	

The following example shows that Proposition \ref{uv=0} does not hold for $n=2$. 

\begin{example}
	Consider the radical square zero algebra $A=\Bbbk Q/J^2$, where
	$$Q:\xymatrix{\bullet\ar@(ul, ur)[]^{\alpha}}$$
	Upon reexamining the proof of Proposition \ref{uv=0} in the context of this case, it becomes evident that the coproducts must exhibit the following structure
	$$\Delta(1) = a\bigl(1\otimes \alpha + \alpha\otimes 1\bigr) + b\bigl(\alpha\otimes\alpha\bigr),\quad a,b\in \k$$ 
	then $\operatorname{Frobdim}_\k(A) = 2 > 1^2$.
\end{example}

The following result characterizes the algebras that satisfy the property presented in Proposition \ref{uv=0}.

\begin{prop}\label{prop radical^2}
	$A$ is a $\k$-algebra with unit $1$ such that  there exists $\mathscr{B}=\bigl\{1, v_2, \dots, v_n\bigr\}$ basis of $A$ as a $\k$-vector space with $v_iv_j=0$ for all $i,j=2, \dots, n$ if and only if $A$ is a radical square  zero algebra with quiver 
	$$Q:
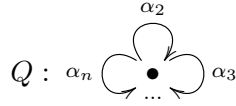
\xymatrix{
		\underset{\dots}{\bullet}  \ar@(ul,dl)_{\alpha_n}\ar@(ur,dr)[]^{\alpha_3}\ar@(ul, ur)[]^{\alpha_2}
		}$$
	$Q_0=\{\bullet_1\}$, $Q_1=\{\alpha_2, \alpha_3,\dots, \alpha_n\}$.	
	
\end{prop}
\begin{proof}
For the converse $\displaystyle{A=\frac{\k Q}{\bigl\langle \alpha_i\alpha_j: i,j=2,\dots, n\bigr\rangle }}$ then $\mathscr{B}=\bigl\{1,\alpha_2, \alpha_3,\dots, \alpha_n\bigr\}$ is a basis of $A$ as a $\k$-vector space with $\alpha_i\alpha_j=0$ for all $i,j=2, \dots, n$.\\

For the direct, we naturally have a bijection $F:A \to \displaystyle{\frac{\k Q}{\bigl\langle \alpha_i\alpha_j: i,j=2,\dots, n\bigr\rangle }}$ 
given by $1\mapsto e_1 $, $v_i\mapsto \alpha_i$ for all $i=2,\dots, n$.\\
Let $u,v\in A$,  $\displaystyle{u=a 1+\sum_{i=2}^na_iv_i}$ and $\displaystyle{v=b 1+\sum_{j=2}^n b_jv_j \Rightarrow uv=ab 1+\sum_{j=2}^n ab_jv_j +\sum_{i=2}^n a_i b v_i}$.

Therefore, $$\displaystyle{F(uv)=F\bigl(ab 1+\sum_{j=2}^n ab_jv_j +\sum_{i=2}^n a_ib v_i\bigr)=ab e_1+\sum_{j=2}^n ab_j\alpha_j+\sum_{i=2}^n a_ib\alpha_i}$$
$$F(u)=F\bigl(a 1+\sum_{i=2}^n a_iv_i\bigr)=a e_1+\sum_{i=2}^n a_i\alpha_i$$
$$F(v)=F\bigl(b 1+\sum_{j=2}^n b_jv_j\bigr)=b e_1+\sum_{j=2}^n b_j\alpha_j$$
$$\Rightarrow F(u)F(v)=\left(a e_1+\sum_{i=2}^n a_i\alpha_i\right)\left(b e_1+\sum_{j=2}^n b_i\alpha_j\right)=ab e_1+ \sum_{j=2}^n
 ab_i\alpha_j+\sum_{i=2}^n a_ib\alpha_i$$
 $$\Rightarrow F(uv)=F(u)F(v)\; \forall u,v\in A.$$

\end{proof}

Combining the previous propositions, we obtain the following result.
\begin{cor}\label{p7}
If $\displaystyle{A=\k Q/R^2}$ is a radical square  zero algebra with quiver $$Q:
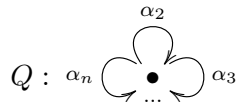
\xymatrix{
		\underset{\dots}{\bullet}  \ar@(ul,dl)_{\alpha_n}\ar@(ur,dr)[]^{\alpha_3}\ar@(ul, ur)[]^{\alpha_2}
	}$$ where
	$Q_0=\{\bullet_1\}$, $Q_1=\{\alpha_2, \alpha_3,\dots, \alpha_n\}$,  then $\operatorname{Frobdim}_\k(A)=(n-1)^2$.
	\end{cor}

\begin{thm}\label{teo7}
	If $A$ is a $\k$-algebra  with unit of finite dimension  $n\geq 3$, then $$\operatorname{Frobdim}_\k(A)\leq (n-1)^2.$$ 	                   
\end{thm}
\begin{proof}
As $A$ is a $\k$-algebra with unit $1$ we can choose $\mathscr{B}=\bigl\{1, v_2, \dots, v_n\bigr\}$ basis of $A$ as a $\k$-vector space.\\	

	Let be $\Delta:A\to A\otimes A$ a nearly Frobenius coproduct, it has the following general expression in 1:
	$$\Delta(1)=\alpha 1\otimes 1+\sum_{j=2}^n\alpha_{j} 1\otimes v_j+\sum_{i=2}^n\beta_{i}v_i\otimes 1+ \sum_{i,j=2}^n\alpha_{ij}v_i\otimes v_j.$$
	We may write an element $v_iv_j\in A$ as a linear combination of $\mathscr{B}$:
	 $$\displaystyle{v_iv_j=d_{ij}^01+\sum_{l=2}^nd_{ij}^lv_l}\quad \mbox{for all}\; i,j\in\{2,\dots,n\}.$$
	If we consider an element  $v_k$ and calculate its coproduct, we obtain that:
	$$\begin{array}{lcl}
		\Delta\bigl(v_k\bigr)=v_k\Delta(1)&=&\displaystyle{\alpha v_k\otimes 1+\sum_{j=2}^n\alpha_{j} v_k\otimes v_j+\sum_{i=2}^n\beta_{i}v_kv_i\otimes 1+ \sum_{i,j=2}^n\alpha_{ij}v_kv_i\otimes v_j}\\
		&=&\displaystyle{\alpha v_k\otimes 1+\sum_{j=2}^n\alpha_{j} v_k\otimes v_j+\sum_{i,l=2}^n\beta_{i}d_{ki}^lv_l\otimes 1+ \sum_{i=2}^{n}\beta_id_{ki}^01\otimes 1}\\
		&&\displaystyle{+\sum_{i,j=2}^n\alpha_{ij}d_{ki}^01\otimes v_j+\sum_{i,j,l=2}^n\alpha_{ij}d_{ki}^lv_l\otimes v_j}
	\end{array}$$
	$$\begin{array}{lcl}
		\Delta\bigl(v_k\bigr)=\Delta(1)v_k&=&\displaystyle{\alpha 1\otimes v_k+\sum_{j=2}^n\alpha_{j}1 \otimes v_jv_k+\sum_{i=2}^n\beta_{i}v_i\otimes v_k+ \sum_{i,j=2}^n\alpha_{ij}v_i\otimes v_jv_k}\\
		&=&\displaystyle{\alpha 1\otimes v_k+\sum_{j,l=2}^n\alpha_{j}d_{jk}^l1\otimes v_l+ \sum_{j=2}^{n}\alpha_jd_{jk}^01\otimes 1 +\sum_{i=2}^n\beta_{i} v_i\otimes v_k}\\
		&&\displaystyle{+\sum_{i,j=2}^n\alpha_{ij}d_{jk}^0v_i\otimes 1+\sum_{i,j,l=2}^n\alpha_{ij}d_{jk}^lv_i\otimes v_l}
	\end{array}$$
	Since $\mathscr{B}\otimes\mathscr{B}$ is a base of $A\otimes A$ we obtain nine types of equations involving the coefficients of $\Delta$.
	\begin{itemize}
		\item Comparing the coefficients of $1\otimes 1$ in both expressions we have that 
		\begin{equation}\label{ecuacion 10}
			\sum_{j=2}^{n}d_{kj}^0\beta_j=\sum_{i=2}^{n}d_{ik}^0\alpha_i
		\end{equation}
		\item Comparing the coefficients of $v_k\otimes 1$ in both expressions we have that 
		\begin{equation}\label{ecuacion 11}
			\alpha=\sum_{j=2}^{n}d_{jk}^0\alpha_{kj}-\sum_{j=2}^{n}d_{kj}^k\beta_j
		\end{equation}
		\item Comparing the coefficients of $v_i\otimes 1$, $i\neq k$, in both expressions we have that 
		\begin{equation}\label{ecuacion 12}
			\sum_{j=2}^{n}d_{kj}^i\beta_j=\sum_{j=2}^{n}d_{jk}^0\alpha_{ij}
		\end{equation}
		\item Comparing the coefficients of $1\otimes v_k$ in both expressions we have that 
		\begin{equation}\label{ecuacion 13}
			\alpha=\sum_{i=2}^{n}d_{ki}^0\alpha_{ik}-\sum_{j=2}^{n}d_{jk}^k\alpha_j
		\end{equation}
		\item Comparing the coefficients of $1\otimes v_j$, $j\neq k$, in both expressions we have that 
		\begin{equation}\label{ecuacion 14}
			\sum_{i=2}^{n}d_{ik}^j\alpha_i=\sum_{i=2}^{n}d_{ki}^0\alpha_{ij}
		\end{equation}
		\item Comparing the coefficients of $v_k\otimes v_k$ in both expressions we have that 
		\begin{equation}\label{ecuacion 15}
			\alpha_k-\beta_k=\sum_{j=2}^{n}d_{jk}^k\alpha_{kj}-\sum_{i=2}^{n}d_{ki}^k\alpha_{ik}
		\end{equation}
		\item Comparing the coefficients of $v_k\otimes v_j$, $j\neq k$, in both expressions we have that 
		\begin{equation}\label{ecuacion 16}
			\alpha_j=\sum_{i=2}^{n}d_{ik}^j\alpha_{ki}-\sum_{i=2}^{n}d_{ki}^k\alpha_{ij}
		\end{equation}
		\item Comparing the coefficients of $v_i\otimes v_k$, $i\neq k$, in both expressions we have that 
		\begin{equation}\label{ecuacion 17}
			\beta_i=\sum_{j=2}^{n}d_{kj}^i\alpha_{jk}-\sum_{j=2}^{n}d_{jk}^k\alpha_{ij}
		\end{equation}
		\item Comparing the coefficients of $v_i\otimes v_j$, $i,j\neq k$, in both expressions we have that 
	\begin{equation}\label{ecuacion 18}
		\sum_{l=2}^{n}d_{kl}^i\alpha_{lj}=\sum_{l=2}^{n}d_{lk}^j\alpha_{il}
	\end{equation}
	\end{itemize}
	Since $n\geq 3$ we can choose $v_r\neq v_k$ and  obtain analogous equations of (\ref{ecuacion 16}) and (\ref{ecuacion 17}) for it, that is
	\begin{equation}\label{ecuacion 19}
		\alpha_j=\sum_{i=2}^{n}d_{ir}^j\alpha_{ri}-\sum_{i=2}^{n}d_{ri}^r\alpha_{ij}\;\mbox{for}\; j\neq r
	\end{equation}
	and
	\begin{equation}\label{ecuacion 20}
		\beta_i=\sum_{j=2}^{n}d_{rj}^i\alpha_{jr}-\sum_{j=2}^{n}d_{jr}^r\alpha_{ij} \;\mbox{for}\; i\neq r
	\end{equation}
	From equations (\ref{ecuacion 13}), (\ref{ecuacion 16}),  (\ref{ecuacion 17}), (\ref{ecuacion 19}) and (\ref{ecuacion 20}) we obtain that $1+(n-1)+(n-1)$ coefficients of $\Delta$ are determined from the others, then
	$$\operatorname{Frobdim}_\k(A)\leq n^2-1-2(n-1)=n^2-2n+1=(n-1)^2.$$
\end{proof}

\begin{thm} \label{frobdim maxima}
	Consider $A$ a $\k$-algebra with unit $1$. There exists $\mathscr{B}=\bigl\{1, v_2, \dots, v_n\bigr\}$ basis of $A$ as a $\k$-vector space with $v_iv_j=0$ for all $i,j=2, \dots, n$, $n\geq 3$ if and only if $\operatorname{Frobdim}_\k(A)=(n-1)^2$.
\end{thm}
\begin{proof}
	The direct statement is exactly Proposition \ref{uv=0}.
	We now prove the converse of the theorem.\\
Using the same notation of the previous theorem,    
	$$\Delta(1)=\alpha 1\otimes 1+\sum_{j=2}^n\alpha_{j} 1\otimes v_j+\sum_{i=2}^n\beta_{i}v_i\otimes 1+ \sum_{i,j=2}^n\alpha_{ij}v_i\otimes v_j$$
	and suppose that there exist $k,s\in\{2,\dots, n\}$ such that $v_kv_s\neq 0$. We will prove that this implies that $\operatorname{Frobdim}_\k(A)<(n-1)^2$.\\
We only need to determine one additional condition over the coefficients of $\Delta$ using the previous restriction.\\
Since  $v_kv_s\neq 0$, we have $\displaystyle{v_kv_s=d_{ks}^01+\sum_{l=2}^nd_{ks}^lv_l}\neq 0$ then $d_{ks}^0\neq 0$ or $d_{ks}^l\neq 0$ for some $l\in\{2,\dots,n\}$.\\

If $d_{ks}^h\neq 0$ for some $h\in\{2,\dots,n\}$ and $h\neq k$, using equation \ref{ecuacion 18}, we have 
$$\alpha_{sj}=\frac{1}{d_{ks}^h}\left(\sum_{l=2}^{n}d_{lk}^j\alpha_{hl}- \sum_{l\neq s}d_{kl}^h\alpha_{lj}\right)\quad\mbox{for all}\; j\neq k.$$

If $d_{ks}^k\neq 0$ we combine \ref{ecuacion 19} and \ref{ecuacion 20} for $i=j=k$ and \ref{ecuacion 15} to determine $\alpha_{rk}$ in terms of the others $\alpha_{ij}$.\\

If $d_{ks}^0\neq 0$ necessarily some $d_{ks}^h\neq 0$, because if $v_iv_j=d_{ij}^01$ for all $i$ and $j$, using the associativity of the product, we would conclude that two elements of the basis are multiples of each other or one is zero.\\
Finally we have that $$\operatorname{Frobdim}_\k(A)<(n-1)^2.$$
\end{proof}


\section{Bounds on the Frobenius dimension for algebras  $\k Q/I$}
Building upon the results of Section 4, we now specialize our analysis to bound quiver algebras. By exploiting the combinatorial structure of path algebras, we derive even sharper bounds for their Frobenius dimension and characterize the precise conditions under which these limits are reached, particularly in quivers with multiple vertices.\\

The following result completely characterizes path algebras of dimension $n$ when their Frobenius dimension is $(n-1)^2$.

\begin{cor}\label{maximal}
	If $\displaystyle{A =\Bbbk Q/I}$ is a connected finite dimensional algebra with $\dim_{\Bbbk} A = n \geq 3 $, then $\operatorname{Frobdim}_\k(A) = (n-1)^2$ if and only if $A$ is a radical square zero algebra and $|Q_0| = 1$.
\end{cor}
\begin{proof}
	Consider $\displaystyle{A =\Bbbk Q/I}$  a connected finite dimensional algebra with $\dim_{\Bbbk} A = n \geq 3 $. If  $\operatorname{Frobdim}_\k(A) = (n-1)^2$ then, by Theorem \ref{frobdim maxima} and Proposition \ref{prop radical^2}, we conclude that $|Q_0| = 1$. If $|Q_1|=1$ we are in the particular case $\displaystyle{A= \frac{\Bbbk[x]}{\bigl(x^{n-1}\bigr)}}$ and $\operatorname{Frobdim}_\k(A) =n$ (see Example 2 from \cite{AGL15}). We can then asume that $|Q_1|\geq 2$.  Let us call $e$ the trivial path and $B$ a basis of $A$. Since the elements of the form $e\otimes \omega$ and $\omega\otimes e$  with $\omega\in B$ cannot appear in $\Delta(1)$, we already know that $$\operatorname{Frobdim}_\k(A) \leq n^2-2n+1=(n-1)^2,$$ so we only need to prove that if $A$ is not radical square zero then the inequality is strict.\\
	Suppose that is not a radical square zero algebra, that means there are at least $2$ loops $\alpha_i$ and $\alpha_j$, not necessarily different, such that $\alpha_i\alpha_j\neq 0$. Then the element $\alpha_j\otimes \alpha_j$ cannot appear in $\Delta(1)$ since $\Delta(\alpha_i)=\Delta(e)\alpha_i=\alpha_i\Delta(e)$ and  $$\operatorname{Frobdim}_\k(A) < (n-1)^2.$$
	
	The converse is Corollary \ref{p7}.
	
\end{proof}

In the following theorem we refine the bound given in Theorem \ref{teo7} for the Frobenius dimension in the case of path algebras with two or more vertices.

\begin{thm}\label{teo}
	Let $\displaystyle{A=\k Q/I}$ be a finite dimensional $\Bbbk$-algebra, where $\operatorname{dim}_\k(A)=n$ and $Q$ is a connected quiver with $|Q_0| = k\geq 2$. 
	Then 
	$$\operatorname{Frobdim}_\k(A) \leq \sum_{i=1}^k g\mathcal{P}^{+}(i).g\mathcal{P}^{-}(i)+2(n-k)-|Q_1|.$$ 
	
\end{thm}

\begin{proof}
	Let $\Delta$ be a nearly Frobenius coproduct. This coproduct admits the following representation:	
	$$\Delta(1) = \sum_{i,j = 1}^{k} \lambda_{ij} e_i \otimes e_j + \sum_{i,j=1}^k\sum_{r=1}^{g\mathcal{P}^{-}(j)} \mu_{ij}^re_i\otimes w_{rj} + \sum_{i,j = 1}^k\sum_{l}^{g\mathcal{P}^{+}(i)} \delta_{ij}^l v_{li}\otimes e_j + \sum_{i,j=1}^k\sum_{r=1}^{g\mathcal{P}^{+}(i)}\sum_{l}^{g\mathcal{P}^{-}(j)}\eta_{ij}^{rl} v_{ri}\otimes w_{lj}$$ 
	where $e_i \in Q_0$, $v_{ki}$ are the paths of positive length starting at the vertex $i$, $w_{lj}$ are the paths of positive length ending at $j$ and $\lambda_{ij}, \mu_{ij}^r, \delta_{ij}^l,\eta_{ij}^{rl} \in \Bbbk$ $\forall i,j,r,l$. 	
	
	Since $\Delta(e_i) = e_i \Delta(1) = \Delta(1)e_i$  $\forall i = 1, \ldots, k $, we have that 
	$$\lambda_{ij} =  \mu_{ij}^r= \delta_{ij}^l=\eta_{ij}^{rl}=0 \text{ if } i\not = j \in Q_0$$ 
	therefore $\Delta$ verifies:
	$$\Delta(1) = \sum_{i= 1}^k \lambda_i \bigl(e_i \otimes e_i\bigr) + \sum_{i=1}^k\sum_r \mu_i^r e_i\otimes w_{ri} + \sum_{i = 1}^k\sum_l \delta_i^l v_{li}\otimes e_i + \sum_{i=1}^{k}\sum_{r,l}\eta_i^{lr} v_{li}\otimes w_{ri}.$$
	
	Since $Q$ is connected, for every vertex $i$ there is an arrow $\alpha:i\to j$ or an arrow $\beta: j\to i$. Consider the first case (the other one is analogous). 
	$$\Delta(\alpha) = \alpha \Delta(1) = \lambda_j \alpha \otimes e_j + \sum_r\mu_j^r\alpha \otimes w_{rj}+\sum_l\delta_j^l  \alpha v_{lj}\otimes e_j +\sum_{l,r}\eta_j^{lr} \alpha v_{lj}\otimes w_{rj}.$$
	$$\Delta(\alpha) = \Delta(1) \alpha = \lambda_i e_i \otimes \alpha + \sum_r\mu_i^r e_i \otimes w_{ri}\alpha + \sum_l\delta_i^l v_{li}\otimes \alpha +\sum_{l,r}\eta_i^{lr} v_{li}\otimes w_{ri}\alpha.$$
	At first sight we observe that $\lambda_i = 0$  $\forall i = 1, \ldots, k$ hence $\Delta$ satisfies that

	$$\Delta(1) = \sum_{i=1}^k\sum_{r} \mu_{i}^re_i\otimes w_{ri} + \sum_{i = 1}^k\sum_{l} \delta_{i}^l v_{li}\otimes e_i + \sum_{i=1}^k\sum_{r,l}\eta_{i}^{rl} v_{ri}\otimes w_{li}.$$ 
	If we consider the subset $\{w_{1j},\cdots ,w_{gr^-(j)j}\}$ of  the paths of length 1 (arrows) ending at $j$ and $\{v_{1i},\cdots ,v_{gr^+(i)i}\}$ the arrows strarting at $i$ then the above equation imposed by $\Delta(\alpha)$ implies that $$ \sum_{r=1}^{gr^-(j)}\mu_j^r\alpha\otimes w_{rj}=\sum_{l=1}^{gr^+(i)}\delta_i^lv_{li}\otimes \alpha$$
	so both sums have at most one non zero summand (let us suppose that is the first one) with  $\mu_j^1=\delta_i^1$ and  $w_{1j}=\alpha=v_{1i}$. Otherwise $\mu_{j}^r=0$ for $r=1,\cdots ,gr^-(j)$.

	In conclusion, every $\mu_{j}^r$ is determined by the $\delta_{i}^l$ or is zero and we have at least $|Q_1|$ independent equations. 
	
		From the above, 
	$$\operatorname{Frobdim}_\k(A)\leq\sum_{i=1}^k g\mathcal{P}^{+}(i).g\mathcal{P}^{-}(i)+2(n-k)-|Q_1|$$ 
	\end{proof}
	\begin{rem}
	Let  us verify that the boundary is actually finer. We will prove that, $$\sum_{i=1}^k g\mathcal{P}^{+}(i).g\mathcal{P}^{-}(i)+2(n-k)-|Q_1|\leq (n-k)^2+2(n-k)-1 < (n-1)^2.$$ 
	First observe that $(n-k)^2+2(n-k)-1 < (n-1)^2$ if and only if $k^2-2(n+1)k +4n-2 < 0 \ \forall k$ that verify the hypothesis. On the other hand the polynomials $p_n(k) = k^2-2(n+1)k +4n-2$ have 2 different roots $r_1 < r_2$ where
	
	\begin{itemize}
		\item $r_1 = n+1-\sqrt{(n-1)^2+2}<2$,
		
		\item $r_2 = n+1+\sqrt{(n-1)^2+2}>2n$.
	\end{itemize}
	
	Since every $k$ that verifies the hypothesis is included in the set $\{2\ldots,n-1\}$, the inequality follows.
\end{rem}

\begin{example}
Consider the radical square zero algebra $A$ associated to the following quiver:
$$
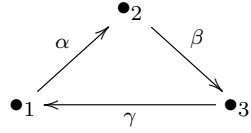
\xymatrix{
	&\bullet_2\ar[dr]^{\beta}&\\
	\bullet_1\ar[ur]^{\alpha}&&\bullet_3\ar[ll]^{\gamma}}$$ 
In this case $\operatorname{Frobdim}_\k(A)=6$, $n=6$, $k=3$, $|Q_1|=3$ and $\sum_{i} g\mathcal{P}^+(i).g\mathcal{P}^{-}(i)=3$ so the boundary is reached.
	\end{example}

\section{Coproduct and Frobenius dimension of truncated path algebras}	

We conclude by investigating the nearly Frobenius structures of truncated path algebras. By introducing the original concepts of paths with no detours and the chop operation, we derive explicit formulas for the coproduct and the exact Frobenius dimension, transforming what was previously a complex computational problem into an elegant and straightforward derivation.

\begin{defn}
Let $Q$ be a quiver. We say that a path in $Q$ has {\bf no detours} if $gr^+(i)=1$ for every vertex involved in the path except maybe the last, and $gr^-(i)=1$ for every vertex except maybe the first.
\end{defn}

For $m\geq 2$,  we define de following sets:
$$\begin{array}{lcl}
	\mathcal{ND}_{m-1}&=&\displaystyle{\bigl\{\alpha_1\cdots\alpha_{m-1}\in Q_{m-1}:\alpha_1\cdots\alpha_{m-1}\; \mbox{has no detours}\bigr\}}\\
	\mathcal{ND}_{m}&=&\displaystyle{\bigl\{\alpha_1\cdots\alpha_{m}\in Q_{m}:\alpha_2\cdots\alpha_{m-1}\; \mbox{has no detours}\bigr\}}\\
	&\vdots&\\
	\mathcal{ND}_{m+i}&=&\displaystyle{\bigl\{\alpha_1\cdots\alpha_{m+i}\in Q_{m+i}:\alpha_{i+2}\cdots\alpha_{m-1}\; \mbox{has no detours}\bigr\}}\\
	&\vdots&\\
	\mathcal{ND}_{2m-2}&=&\displaystyle{\bigl\{\alpha_1\cdots\alpha_{2m-2}\in Q_{2m-2}\bigr\}}
\end{array}$$

\begin{defn}
If $\alpha_1\cdots\alpha_t$ is a path of the quiver $Q$, we define its {\bf chop} as 
$$w_{\alpha_1\cdots\alpha_t}=\sum_{s:\max\{s,t-s\}<m}\alpha_{s+1}\cdots\alpha_{t}\otimes \alpha_1\cdots\alpha_{s}$$
\end{defn}

\begin{thm}\label{frobdim truncada LLm}
	Let $\displaystyle{A=\k Q/R^m}$ be a truncated path $\Bbbk$-algebra,  and $Q$ a connected quiver. 
	Then, the general expression of the nearly Frobenius coproduct is:
	$$\Delta(1)=\sum_{i=-1}^{m-2}\quad\sum_{\alpha_1\cdots\alpha_{m+i}\in \mathcal{ND}_{m+i}}\lambda_{\alpha_1\cdots\alpha_{m+i}}w_{\alpha_1\cdots\alpha_{m+i}}$$
	where $\lambda_{\alpha_1\cdots\alpha_{m+i}}\in\k$,  and $w_{\alpha_1\cdots\alpha_{m+i}}$ is the chop of  $\alpha_1\cdots\alpha_{m+i}$, for $i=-1,\dots, m-2$.\\
		In particular,
			\begin{equation}\label{e2}
			\operatorname{Frobdim}_\k(A)= \sum_{i=-1}^{m-2}|\mathcal{ND}_{m+i}|=\sum_{i=-1}^{m-3}|\mathcal{ND}_{m+i}|+|Q_{2m-2}|	
			\end{equation}
\end{thm}
\begin{proof}
	We will describe what conditions the elementary tensors derived from a path in the quiver must meet to appear in the nearly Frobenius coproduct.\\
		
	Using the same arguments as in the proof of Theorem \ref{teo} we have that terms of the form $e_i \otimes e_i$ cannot appear in the coproduct and that if $u\otimes v$ appears in $\Delta(1)$ then $s(u)=t(v)$.
	Then, if $\vert Q_0\vert=k$, we can express the coproduct as follows:
	$$\Delta(1) = \sum_{i=1}^k\sum_{r} \mu_{i}^re_i\otimes w_{ri} + \sum_{i = 1}^k\sum_{l} \delta_{i}^l v_{li}\otimes e_i + \sum_{i=1}^k\sum_{r,l}\eta_{i}^{rl} v_{ri}\otimes w_{li},$$ 
	where $v_{ri}$ are the paths of positive length starting at the vertex $i$ and $w_{lj}$ are the paths of positive length ending at $j$.
	
	Let us now suppose that $u\otimes v$ appears in $\Delta(1)$ with coefficient $\nu$ with $u=\alpha_{r+1}\cdots \alpha_t$ and $v=\alpha_1\cdots \alpha_r$. Then $t-r\leq m-1$ and $r\leq m-1$. Since
	$\Delta(\alpha_r)=\alpha_r\Delta(1)=\Delta(1)\alpha_r$, in the case where $t+1-r\leq m-1$, we obtain that the term $\alpha_r\alpha_{r+1}\cdots \alpha_t\otimes \alpha_1\cdots \alpha_{r-1}$ appears in $\Delta(1)$ with the same coefficient $\nu$.  If $r+1\leq m-1$ we can proceed in an analogous way and obtain that $\alpha_{r+2}\cdots \alpha_t\otimes \alpha_1\cdots \alpha_{r+1}$ appears in $\Delta(1)$ with the same coefficient $\nu$. In conclusion, if $\alpha_{r+1}\cdots \alpha_t\otimes \alpha_1\cdots \alpha_{r}$ appears in $\Delta(1)$ with coefficient $\nu$, then the following sum appears in $\Delta(1)$:
	\begin{equation}\label{paseos}
		\nu\left(\sum_{s:\max\{s,t-s\}<m}\alpha_{s+1}\cdots\alpha_{t}\otimes \alpha_1\cdots\alpha_{s}\right)
	\end{equation}
	Let us now consider the terms of the form $\alpha_1\cdots \alpha_t\otimes e_i$ with $t<m-1$. We will prove that these terms cannot appear in $\Delta(1)$: 
	
	First consider the case were $\alpha_1\cdots \alpha_t$ does not contain a cycle.
	Since $Q$ is connected and $R^m\neq \{0\}$, there exists a vertex $j$ in the path $\alpha_1\cdots \alpha_t$ and an arrow $\beta\neq \alpha_i$ $i=1,\cdots ,t$ such that $t(\beta)=j$ or $s(\beta)=j$.\\	
	Let us suppose that $t(\beta)=t(\alpha_k)$ for some $k=1,\dots, t $.
	If $\mu\alpha_1\cdots \alpha_t\otimes e_i$ appears in $\Delta(1)$, using the argument previously exposed, $\mu\alpha_{k+1}\cdots \alpha_t\otimes \alpha_1\cdots \alpha_k$ appears in $\Delta(1)$ and $\mu\beta\alpha_{k+1}\cdots \alpha_t\otimes \alpha_1\cdots \alpha_k$ would appear in $\Delta(\beta)=\beta\Delta(1)=\Delta(1)\beta$. But this cannot happen since $\alpha_k\neq \beta$. This means that $\mu=0$ and $\alpha_1\cdots \alpha_t\otimes e_{i}$ cannot appear in $\Delta(1)$. 
	The case where  $s(\beta)=s(\alpha_k)$, $s(\beta)=t(\alpha_t)$ and $t(\beta)=s(\alpha_1)$ are analogous.
\begin{figure}[h]
	\centering
	\[\begin{tikzcd} 
	\bullet & \bullet & \cdots & \bullet & \bullet &&&& \\ 
	&& \bullet &&& \bullet \\ && \vdots &&& \vdots \\ && \bullet &&& \bullet & \bullet & \cdots & \bullet \\ 
	&&& \bullet & \bullet \arrow["{\gamma_1}", from=1-1, to=1-2] \arrow["{\gamma_2}", from=1-2, to=1-3] \arrow["{\gamma_q}", from=1-3, to=1-4] \arrow["{c_1}", from=1-4, to=1-5] \arrow["{c_2}", from=1-5, to=2-6] \arrow["{c_p}", from=2-3, to=1-4] \arrow[from=2-6, to=3-6] \arrow[from=3-3, to=2-3] \arrow["{c_r}", from=3-6, to=4-6] \arrow[from=4-3, to=3-3] \arrow["{\beta_1}", from=4-6, to=4-7] \arrow[from=4-6, to=5-5] \arrow["{\beta_2}", from=4-7, to=4-8] \arrow["{\beta_s}", from=4-8, to=4-9] \arrow[from=5-4, to=4-3] \arrow[from=5-5, to=5-4]
	 \end{tikzcd}\]
	 \caption{cycle in the path}
	  \label{ciclo}
\end{figure}	
	If $\alpha_1\cdots\alpha_t$ contains a cycle $c_1\cdots c_p$ then is of the form $\gamma_1\cdots \gamma_q (c_1\cdots c_p)^h c_1\cdots c_r\beta_1\cdots \beta_s$, where $\gamma_q\neq c_p$ and $\beta_1\neq c_{r+1}$ (see Figure \ref{ciclo}). We will divide the study into two cases. 
	
	\begin{itemize}
	\item If $\gamma_1\cdots \gamma_q=0$, then $\alpha_1\cdots\alpha_t$ is of the form $(c_1\cdots c_p)^h c_1\cdots c_r\beta_1\cdots \beta_s$. 	If $\mu\alpha_1\cdots \alpha_t\otimes e_i$ appears in $\Delta(1)$, then $\mu c_p(c_1\cdots c_p)^h c_1\cdots c_r\beta_1\cdots \beta_s\otimes e_i$ would appear in $\Delta(c_p)=c_p\Delta(1)=\Delta(1)c_p$. But this cannot happen since $c_p\neq e_i$. 
	
\item If $\gamma_1\cdots \gamma_q\neq 0$, we suppose   that $\mu \gamma_1\cdots \gamma_q (c_1\cdots c_p)^h c_1\cdots c_r\beta_1\cdots \beta_s\otimes e_i$ 
	appears in $\Delta(1)$, then $\mu  (c_1\cdots c_p)^h c_1\cdots c_r\beta_1\cdots \beta_s\otimes \gamma_1\cdots \gamma_q$ appears too. Therefore,\\  $\mu  c_p(c_1\cdots c_p)^h c_1\cdots c_r\beta_1\cdots \beta_s\otimes \gamma_1\cdots \gamma_q$
	appears in $\Delta(c_p)=\Delta(1)c_p$, but, as before, using that $\gamma_q\neq c_p$ we conclude that $\mu=0$.
\end{itemize}
	Then, using the equation (\ref{paseos}) we conclude that the elements of the form $u\otimes v$ such that the sum of the lengths of $u$ and $v$ is lower than $m-1$ will neither appear in $\Delta(1)$.
	
	Then
	$$\Delta(1)=\sum_{i=-1}^{m-2}\quad\sum_{\alpha_1\cdots\alpha_{m+i}\in Q_{m+i}}\lambda_{\alpha_1\cdots\alpha_{m+i}}w_{\alpha_1\cdots\alpha_{m+i}}$$
	where $\lambda_{\alpha_1\cdots\alpha_{m+i}}\in k$.

	We will prove then that the coproduct has the following expression
	$$\Delta(1)=\sum_{i=-1}^{m-2}\quad\sum_{\alpha_1\cdots\alpha_{m+i}\in \mathcal{ND}_{m+i}}\lambda_{\alpha_1\cdots\alpha_{m+i}}w_{\alpha_1\cdots\alpha_{m+i}}$$
	Observe that, if $t=m+i$ then 
	\begin{equation}\label{chop}
		w_{\alpha_1\cdots\alpha_{m+i}}=\sum_{s=i+1}^{m-1}\alpha_{s+1}\cdots\alpha_{m+i}\otimes\alpha_1\cdots \alpha_{s}
	\end{equation}
	To prove that the paths involved in the chops in $\Delta(1)$ belong to $\mathcal{ND}_{m+i}$ for all $-1\leq i\leq m-2$ let us suppose by absurd that $\alpha_1\cdots \alpha_{m+i}\in Q_{m+i}$ with $\alpha_{i+2}\cdots \alpha_{m-1}$ having a detour. Consider the case where $gr^-(t(\alpha_{k}))\geq 2$  with $i+2\leq k\leq m-1$ and name $\beta\neq \alpha_k$ such that $t(\beta)=t(\alpha_k)$. If $\alpha_1\cdots \alpha_{m+i}$ appears in a chop in $\Delta(1)$ then $\alpha_{k+1}\cdots \alpha_{m+i}\otimes \alpha_1\cdots \alpha_k\in \Delta(1)$ and $\beta\alpha_{k+1}\cdots \alpha_{m+i}\otimes \alpha_1\cdots \alpha_k\in \Delta(1)$ should belong to $\Delta(\beta)$ but this cannot happen since it doesn't end in $\beta$. Then $gr^{-}(t(\alpha_k))=1$ and with an analogous argument we can prove that  $gr^{+}(s(\alpha_k))=1$ and then $\alpha_1\cdots \alpha_{m+i}\in \mathcal{ND}_{m+i}$.\bigbreak
	Let us verify now that every chop of the form $$w_{\alpha_1\cdots \alpha_{m+i}}=\sum_{s=i+1}^{m-1}\alpha_{s+1}\cdots\alpha_{m+i}\otimes\alpha_1\cdots \alpha_{s}$$ with  $\alpha_1\cdots\alpha_{m+i}\in \mathcal{ND}_{m+i}$ satisfies the Frobenius condition, let us suppose that $\Delta(1)=w_{\alpha_1\cdots\alpha_{m+i}}$.
	Given $\alpha\in Q_1$ let us check that $\Delta(\alpha)=\alpha\Delta(1)=\Delta(1)\alpha$ considering the following cases:
	\begin{itemize}
	\item  If $\alpha$ is not connected to the path $\alpha_{i+2}\cdots\alpha_{m-1}$ then using the expression  \ref{chop}, we conclude that $\Delta(\alpha)=0$ and the condition holds.	
	\item If $t(\alpha)=s(\alpha_{i+2})$  (the case $s(\alpha)=t(\alpha_{m-1})$ is similar) then  $\Delta(\alpha)=\alpha\Delta(1)=\alpha\alpha_{i+2}\cdots \alpha_{m+i}\otimes \alpha_1\cdots \alpha_{i+1}=0$ since  $\alpha\alpha_{i+2}\cdots \alpha_{m+i}$ has length $m$. On the other hand, $\Delta(1)\alpha=\alpha_{i+2}\cdots \alpha_{m+i}\otimes \alpha_1\cdots \alpha_{i+1}\alpha$ and if $\alpha_1\cdots \alpha_{i+1}\alpha\neq 0$ then $s(\alpha)=t(\alpha_{i+1})=s(\alpha_{i+2})$ so if $\alpha$ is not a loop this cannot happen. In the case $\alpha$ is a loop, since $\alpha_{i+2}\cdots \alpha_{m-1}$ has no detours, the only option is one vertex with one loop and in this particular case the Frobdim was already computted (see Example 2 from \cite{AGL15}) and the result holds. 
\item If $\alpha$ belongs to $\alpha_{i+2}\cdots\alpha_{m-1}$, this means $\alpha=\alpha_k$ for some $k=i+2,\cdots m-1$ and $\Delta(\alpha_k)=\alpha_k\Delta(1)=\Delta(1)\alpha_k=\alpha_k\alpha_{k+1}\cdots \alpha_{m+i}\otimes \alpha_1\cdots\alpha_k$.

	\end{itemize}

\end{proof}

\begin{cor}\label{frobdim radical^2=0}
	Let $\displaystyle{A=\k Q/R^2}$ be a finite dimensional radical square zero $\Bbbk$-algebra, where $Q$ is a connected quiver. Then, the general expression of the nearly Frobenius coproduct is:
	$$\Delta(1) = \sum_{\alpha\in  \mathcal{ND}_{1}} \lambda_\alpha\bigl(\alpha\otimes e_{s(\alpha)}+e_{t(\alpha)}\otimes\alpha\bigr) + \sum_{\alpha_1\alpha_2\in Q_2}\lambda_{\alpha_1\alpha_2}\,\alpha_2\otimes\alpha_1.$$ 
	where $\lambda_\alpha, \lambda_{\alpha_1\alpha_2} \in\Bbbk$ for all $\alpha\in  \mathcal{ND}_{1}$ and $\alpha_1\alpha_2\in Q_2$.\\
	In particular,
	\begin{equation}\label{e1}
		\operatorname{Frobdim}_\k(A)=|\mathcal{ND}_{1}|+|Q_2|.
		\end{equation}
	\end{cor}

\begin{proof}
	Observe that 
	$$\omega_\alpha=\alpha\otimes e_{s(\alpha)}+e_{t(\alpha)}\otimes\alpha $$ and $$\omega_{\alpha_1\alpha_2}=\alpha_2\otimes\alpha_1$$

\end{proof}

We will finish this paper by studying the case of regular quivers.

\begin{cor}\label{regular quivers}
Let $\displaystyle{A=\k Q/R^m}$ be a finite dimensional truncated path $\Bbbk$-algebra, where $Q$ is a regular connected quiver, then
\begin{enumerate}
\item[$(a)$] If $\operatorname{gr}(Q) = 1$: $\displaystyle{\Delta(1)=\sum_{i=-1}^{m-2}\quad\sum_{\alpha_1\cdots\alpha_{m+i}\in Q_{m+i}}\lambda_{\alpha_1\cdots\alpha_{m+i}}w_{\alpha_1\cdots\alpha_{m+i}}}$ and 	
$$\operatorname{Frobdim}_\k(A) = m|Q_0|.$$
\item[$(b)$] If $\operatorname{gr}(Q)=k$:  $\displaystyle{\Delta(1)=\sum_{\alpha_1\cdots\alpha_{2m-2}\in Q_{2m-2}}\lambda_{\alpha_1\cdots\alpha_{2m-2}}w_{\alpha_1\cdots\alpha_{2m-2}}}$ and 
$$\operatorname{Frobdim}_\k(A) = k^{2m-2}|Q_0|.$$
\end{enumerate}	
\end{cor}
\begin{proof}
	\begin{enumerate}
		\item[$(a)$] Note that, if  $\operatorname{gr}(Q) = 1$, $\mathcal{ND}_{m+i}=Q_{m+i}$, for $i=-1,\dots, m-2$. Then, the coproduct, applying Theorem  \ref{frobdim truncada LLm}, is
		  $\displaystyle{\Delta(1)=\sum_{i=-1}^{m-2}\quad\sum_{\alpha_1\cdots\alpha_{m+i}\in Q_{m+i}}\lambda_{\alpha_1\cdots\alpha_{m+i}}w_{\alpha_1\cdots\alpha_{m+i}}}$ and 
		  $\displaystyle{\operatorname{Frobdim}_\k(A)= \sum_{i=-1}^{m-2}|Q_{m+i}|}$. Since $\operatorname{gr}(Q) = 1$ we have that $|Q_{m+i}|=|Q_0|$, then $\displaystyle{\operatorname{Frobdim}_\k(A)=m|Q_0|}$.
		  \item[$(b)$] In this case $\mathcal{ND}_{m+i}=0$ for $i=-1,\dots, m-3$, since we do not have no detours. Then, by Theorem  \ref{frobdim truncada LLm},
		  $$\displaystyle{\Delta(1)=\sum_{i=-1}^{m-2}\quad\sum_{\alpha_1\cdots\alpha_{m+i}\in \mathcal{ND}_{m+i}}\lambda_{\alpha_1\cdots\alpha_{m+i}}w_{\alpha_1\cdots\alpha_{m+i}}=\sum_{\alpha_1\cdots\alpha_{2m-2}\in Q_{2m-2}}\lambda_{\alpha_1\cdots\alpha_{2m-2}}w_{\alpha_1\cdots\alpha_{2m-2}}}$$
		    In particular, using that $\operatorname{gr}(Q)=k$, $\operatorname{Frobdim}_\k(A) = |Q_{2m-2}|= k^{2m-2}|Q_0|$.
		    
	\end{enumerate}
\end{proof}

Since the algebras in Item $(a)$ from the corollary above are Frobenius algebras, the result can also be proved using Corollary 10 of \cite{AGM22}.

If we consider the particular case where the quiver has a single vertex, we recover the truncated polynomial algebra in \(n\) variables, where \(n\) is the number of arrows of the quiver. We see how Corollary \ref{regular quivers} looks in this context.

\begin{cor}
	\begin{enumerate}
		\item If $\displaystyle{A=\frac{\k\bigl[x\bigr]}{\bigl\langle x^m\bigr\rangle}}$, then		
	 $\displaystyle{\Delta(1)=\sum_{i=-1}^{m-2}\lambda_i\left(\sum_{j+k=m+i: j,k\leq m-1}x^j\otimes x^k\right)}$ and\; $\operatorname{Frobdim}_\k(A)=m$.
	
		\item If $\displaystyle{A=\frac{\k\bigl[x_1,\dots, x_n\bigr]}{\bigl\langle x_1^{i_1}\cdots x_n^{i_n}: i_1+\cdots i_n=m\bigr\rangle}}$, then
		\begin{itemize}
			\item $\displaystyle{\Delta(1)=\sum_{i_1+\cdots+i_n=2m-2}\lambda_{i_1,\dots,i_n}\left(\sum_{j_s,k_s: j_s+k_s=i_s, s=1,\dots,n}x_1^{j_1}\cdots x_n^{j_n}\otimes x_1^{k_1}\cdots x_n^{k_n}\right)}$
			\item $\operatorname{Frobdim}_\k(A)=n^{2(m-1)}$.
		\end{itemize}	
	\end{enumerate}

\end{cor}

\section*{Acknowledgments}
This work emerged from our participation in the workshop Matemáticas en el Cono Sur 2, which took place in Montevideo in February 2024. We would like to express our gratitude to the organizers of the event: Eugenia Ellis, Maria Inés Fariello, and Andrea Solotar, for their efforts in making this collaboration possible.

\bibliographystyle{elsarticle-harv}

\begin{thebibliography}{}
	

\bibitem{A96} L. Abrams, {\em Two-dimensional topological quantum field theories and Frobenius algebras}. J. Knot Theory Ramifications 5, no. 5,pp. 569-587 (1996).
	
\bibitem{AGL15} D. Artenstein, A. Gonz\'alez,  and M. Lanzilotta, {\em {C}onstructing {N}early {F}robenius {A}lgebras}. Algebras and Representation Theory (2015) {V}olume 18, 339-367.

\bibitem{AGM20} D. Artenstein, A. Gonz\'alez,  and G. Mata  {\em Nearly Frobenius structures in some families of algebras}. São Paulo J. Math. Sci.14(2020), no.1, 165-184.


\bibitem{AGM22} D. Artenstein, A. González, and G. Mata, {\em Nearly Frobenius dimension of Frobenius algebras}. Communications in Algebra, 50(11), 4906–4916 (2022). https://doi.org/10.1080/00927872.2022.2077953


\bibitem{Assem} I.  Assem, D.  Simson and A.  Skowro{\'n}ski.  {\em Elements of the Representation Theory of Associative Algebras}.  Cambridge University Press (2006).

https://doi.org/10.1080/00927870008826917.


\bibitem{CG} R. L. Cohen, V. Godin, A polarized view of string topology. Topology, geometry and
quantum field theory, London Math. Soc. Lecture Note Ser., vol. 308, pp. 127-154 (2004).

\bibitem{GLSU} A. González, E. Lupercio, C. Segovia, B. Uribe, Nearly Frobenius algebras. European Journal of Mathematics 5, 881–902 (2019). https://doi.org/10.1007/s40879-019-00354-3.

\bibitem{KST21} Kobayashi, Yuji; Shirayanagi, Kiyoshi; Tsukada, Makoto and Takahasi, Sin-Ei, {\emph A complete classification of three-dimensional algebras over $\mathbb{R}$ and $\mathbb{C}$ }. Asian-European Journal of Mathematics. Volume 14(8),  (2021), 		2150131 (25 pages). 
		https://doi.org/10.1142/S179355712150131X


\bibitem{SY} A. Skowro{\'n}ski, K. Yamagata, \emph{Frobenius algebras. I. Basic representation theory}. EMS Textbooks in Mathematics. European Mathematical Society (EMS), Zürich (2011).


\end{thebibliography}

\end{document}